\documentclass{cmslatex}
\usepackage[paperwidth=7in, paperheight=10in, margin=.875in]{geometry}
 \usepackage[backref,colorlinks,linkcolor=red,anchorcolor=green,citecolor=blue]{hyperref}
\usepackage{amsfonts,amssymb}
\usepackage{amsmath}
\usepackage{graphicx}
\usepackage{cite}
\usepackage{enumerate}
\renewcommand{\theequation}{\arabic{section}.\arabic{equation}}

   \allowdisplaybreaks
\begin{document}
 \title{Periodic solutions and the avoidance of pull--in instability in non--autonomous micro--electro--mechanical systems\thanks{Received date, and accepted date (The correct dates will be entered by the editor).}}

          %For each author, make a block with the following macros:

          \author{Shirali Kadyrov\thanks{Suleyman Demirel University, Kaskelen, Kazakhstan, (shirali.kadyrov@sdu.edu.kz).}
          \and Ardak Kashkynbayev \thanks{Department of Mathematics, Nazarbayev University, Kabanbay-Batyr 53, 010000 Nur-Sultan, Kazakhstan, (ardak.kashkynbayev@nu.edu.kz).}
          \and  Piotr Skrzypacz  \thanks{Department of Mathematics, Nazarbayev University, Kabanbay-Batyr 53, 010000 Nur-Sultan, Kazakhstan, (piotr.skrzypacz@nu.edu.kz).}
          \and Konstantinos Kaloudis \thanks{Department of Mechanical and Aerospace Engineering, Nazarbayev University, Kabanbay Batyr 53, 010000 Nur-Sultan, Kazakhstan, (konst.kaloudis@gmail.com).}
          \and Anastasios Bountis\thanks{Department of Mathematics, University of Patras Patras, 26500 Greece, (anastasios.bountis@nu.edu.kz).}
          }

         \pagestyle{myheadings} \markboth{SHORT FORM FOR RUNNING HEADS}{AUTHORS' NAME} \maketitle

          \begin{abstract}
             We study periodic solutions of a one-degree of freedom micro-electro-mechanical system (MEMS) with a parallel-plate capacitor under $T$--periodic electrostatic forcing. We obtain analytical results concerning the existence of $T-$ periodic solutions of the problem in the case of arbitrary nonlinear restoring force, as well as when the moving plate is attached to a spring fabricated using graphene. We then demonstrate numerically on a $T-$ periodic Poincar{\'e} map of the flow that these solutions are generally locally stable with large ``islands'' of initial conditions around them, within which the pull-in stability is completely avoided. We also demonstrate graphically on the Poincar{\'e} map that stable periodic solutions with higher period $nT, n>1$ also exist, for wide parameter ranges, with large ``islands'' of bounded motion around them, within which all initial conditions avoid the pull--in instability, thus helping us significantly increase the domain of safe operation of these MEMS models. 
          \end{abstract}

          \section{Introduction}\label{S:Introduction}
Among various actuators, electrostatic ones are among the most commonly used micro-electro-mechanical systems (MEMS) and nano-electro-mechanical systems (NEMS), as they have many advantages including fast response time, small DC power consumption, compatible integration and fabrication \cite{PS13}. In the present paper, the proposed electrostatically actuated MEMS are realized as mass-spring models with two parallel capacitor plates where the one attached to the spring is movable, and the other is kept fixed, while a periodically varying voltage is applied, as shown in Figure~\ref{fig_1}. Here, $d$ is the distance between the plates, $k$ is the stiffness constant of the spring, $x \in [0,d]$ is the displacement moving plate of mass $m$, $c$ is the viscous damping coefficient, $\varepsilon_0>0$ is the electric emissivity and finally $A$ is the area of the attracting parallel plates.

\begin{figure}[htb]
	\begin{center}
		\includegraphics[scale=0.5]{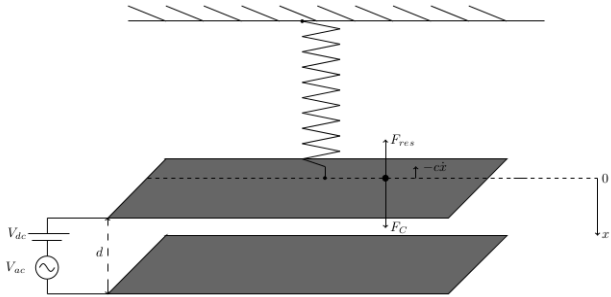}
		\caption{The parallel-plate capacitor.}\label{fig_1}
	\end{center}
\end{figure}

This model was discovered in 1967 by Nathanson et al. \cite{NNWD67} and has been studied since by many researchers. For more recent studies see e.g. \cite{SHEK20, SKNW19, younis2019, GLXMZJZ18, WKK17, Flo17, ZYPM14, IY10, AP07} and  references therein.
When nonlinear stiffness constants of the plates are ignored in the constitutive equation, the oscillations of the system can be modeled by the second-order singular equation 
\begin{equation}\label{sing_osc}
m \ddot{x}(t)+c \dot{x}(t)+h(x)=F_C=\frac{\varepsilon_0A V^2(t)}{2 \bigl(d-x(t)\bigr)^2}\,.
\end{equation}
The above equation of motion includes Coulomb's inverse-square law for the electrostatic force $F_C$,a stiffness function $h(x)$ that reduces to Hooke's law for a linear spring, and a viscous dissipation force linearly proportional to the speed of the moving plate, i.e., $c\dot{x}$. 

We consider here a periodically varying AC-DC voltage of the form
\begin{equation}\label{eq_ACDC}
V(t)=V_{dc}+V_{ac} \cos(\omega t)=\lambda\bigl(1+\delta\cos{(\omega t)}\bigr),
\end{equation} 
with period $T=2\pi/\omega$, where $\lambda=V_{dc}$, $\delta=\frac{V_{ac}}{V_{dc}}$, and assume $|\delta|<1$.

In this case, the system becomes non-autonomous, while the singular term on the right-hand side makes the analysis particularly challenging. Recently, this model was considered in \cite{GT13} to study saddle--node bifurcations and analyze the existence and stability of $T-$ periodic solutions. In \cite{shang2017}, the thresholds of AC voltage for pull-in instability in the initial system and the feedback controlled systems were obtained analytically by the Melnikov method \cite{AH07}.

Perhaps, the most important phenomenon in all these actuator models is their so called pull-in instability effect, whereby, when the applied voltage exceeds a certain critical value the moving plate collapses onto the fixed one. In this case, we say that the pull-in has occurred and the critical value of the voltage is called the pull-in voltage. Clearly, the singular term in the equation and the presence of a dissipation term make it difficult to compute the exact value of the pull-in voltage. This is very important as, in practice, one wants to have bounded oscillations and avoid instability effects when designing micro-resonators. 

One way to achieve this is to require that \eqref{sing_osc} has stable periodic solutions with large regions around them where the motion is confined for $t\rightarrow \infty$ avoiding the pull-in instability.  
For the above  mass-spring system, Zhang et al. \cite{ZYPM14} describe the dynamic pull-in as the collapse of the moving structure caused by the combination of kinetic and potential energies. In general, the dynamic pull-in requires a lower voltage to be triggered compared to the static pull-in threshold, see \cite{Flo17, ZYPM14}. Unfortunately, there are few analytical results for mass-spring models based on nonlinear restoring forces. 

Thus, in Section\ref{per_rig} of the paper we study the existence of periodic solutions for the model with the same period $T$ as the time-dependent voltage using elementary analytical tools, so that we may estimate pull-in voltages for this model for a general nonlinear stiffness function $h(x)$ in (\ref{sing_osc}). Next, in Section \ref{graphene} we formulate the mathematical model for a MEMS  parallel plate attached to a spring made of graphene, obtain analytical results for its $T-$ periodic solutions and analyze the properties of its pull-in instabilities.

In Section \ref{Poincare} we perform numerical computations, using the Poincar{\'e} map of the flow. As a first step, in the absence of dissipation ($c=0$), we demonstrate the stability of our $T-$ periodic solution, located at the origin of Poincar{\'e} map, by graphically depicting the existence of large domains of bounded oscillations around it, for wide ranges of parameter values. We then demonstrate on the Poincar{\'e} map the existence of stable periodic solutions of period $nT$, $n>1$, further away from the $T-$ periodic solution, which also possess large ``islands'' around them, within which the motion remains bounded away from the pull--in instability. Finally, for very small dissipation,($0<c<<1$), the period $T$ solution at the origin becomes a global attractor, while stable $nT-$ periodic solutions still exist and can attract initial conditions in their vicinity. Our conclusions are described in Section \ref{conclusions}.

\section{T-Periodic solutions: Rigorous results}\label{per_rig}

\noindent We begin our study of equation \eqref{sing_osc} by investigating analytically the existence of periodic solutions having the same period $T$ as the periodic forcing term \eqref{eq_ACDC}. To this end, we define 
\begin{equation*} 
V_m=\min\limits_{t\ge 0} |V(t)|
\quad\text{and}\quad 
V_M=\max\limits_{t\ge 0} |V(t)|\,,
\end{equation*} 
and state the following result:
\begin{theorem} \label{thm:main}
	Let $C$ be a non-negative constant, $V(t)$ is a periodic real function, $h(x)$ a continuous real function on $(-\infty, 1]$. Then, the second order nonlinear differential equation
	\begin{equation} \label{eqn:main}
	\ddot x + C\dot x  +h(x) - \frac{V^2(t)}{(1-x)^2}=0,
	\end{equation}
	admits a periodic solution provided the equation
	$$h(x)=\frac{(V_M)^2}{(1-x)^2}$$
	has a root in $[0,1)$. On the other hand, equation \eqref{eqn:main} admits no periodic solutions provided the equation
	$$h(x)=\frac{V_m^2}{(1-x)^2}$$
	has no roots in $(-\infty,1)$, but $h(x)$ has at least one real root on $(-\infty, 1).$
\end{theorem}

\begin{proof}
	Following \cite{DH06}, a periodic function $x(t)$ is said to be \emph{a lower solution} if 
	$$\ddot x + C \dot x + h(x) - \frac{V^2(t)}{(1-x)^2} \ge 0$$ 
	and \emph{an upper solution} if 
	$$\ddot x + C \dot x + h(x) - \frac{V^2(t)}{(1-x)^2} \le 0$$ 
	for all $t\ge 0$.
	We infer from \cite[Chp 1, Theorem~1.1]{DH06} that there exists a periodic solution if there are lower and upper solutions $ x(t), y(t) $ respectively such that $x(t) < y(t)$ for all $t$. Let $x_\ell$ be any root of the equation 
	\[
	h(x)=(V_M)^2/(1-x)^2
	\] 
	in $[0,1]$. Then, clearly $x_\ell<1$ and $x(t)=x_\ell$ satisfies
	$$\ddot x + k \dot x + h(x) - \frac{V^2(t)}{(1-x)^2} =h(x_\ell)-\frac{V^2(t)}{(1-x_\ell)^2} \ge h(x_\ell)-\frac{(V_M)^2}{(1-x_\ell)^2} =0,$$
	so that $x(t)=x_\ell$ is a lower solution. Since 
	$$h(1)(1-1)^2=0 \le V_m^2 \le (V_M)^2=h(x_\ell)(1-x_\ell)^2,$$ 
	it follows from the Intermediate Value Theorem for continuous functions that there exists $x_u \in [x_\ell,1]$ such that  $h(x_u)(1-x_u)^2=V_m^2 \le V^2(t)$. Thus, $x(t)=x_u$ is an upper solution such that $x_\ell \le x_u$, which gives the first part of the proof.
	
	If the system admits a periodic solution $x(t)$ then this function must have a maximum. This means that there exists $t=t_0$ such that $\ddot x(t_0) \le 0$ and $\dot x(t_0)=0$. Then, we conclude that 
	$$0=\ddot x(t_0) + C \dot x(t_0) + h(x(t_0)) - \frac{V^2(t_0)}{(1-x(t_0))^2} \le h(x(t_0)) - \frac{V^2(t_0)}{(1-x(t_0))^2}\,,$$
	and hence
	\begin{equation}\label{eqn:no}
	h(x(t_0))  \ge \frac{V_m^2}{(1-x(t_0))^2}\,.
	\end{equation}
	On the other hand, by assumption, $h(x)=V_m^2/(1-x)^2$ has no roots in $(-\infty,1)$. Since $h$ is continuous it follows that either $h(x)>V_m^2/(1-x)^2$ or $h(x)<V_m^2/(1-x)^2$ on $(-\infty,1)$. The former is not possible as $h(x)$ has a zero in $(-\infty, 1).$ Hence, we must have $h(x)<V_m^2/(1-x)^2$ on $(-\infty,1)$ which contradicts \eqref{eqn:no}. 
\end{proof}

%==========================================
% Mathematical model
%============================================================================================================================
\section{The case of the graphene oscillator}\label{graphene}
%===================================================================================%=========================================
\subsection{Mathematical formulation}

\noindent To motivate Theorem~\ref{thm:main}, we now describe our graphene oscillator model in some detail and recall how the nonlinear equation is derived. We also state other related results. As is well--known, graphene is an ideal material for implementation in NEMS and graphene based materials are being used in practice for designing actuators by material scientists due to its low density and high strength \cite{HJY12}.

Both experimental and theoretical studies \cite{LWKH08, LH09} suggest that the graphene material obeys the following constitutive equation
$$\sigma=E \varepsilon - D |\varepsilon|\varepsilon,$$
where $\varepsilon$, $\sigma$, $E$ and $D$ are the axial strain, axial stress, Young's modulus and second-order elastic stiffness constant, respectively. 

In the mass-spring model above, the restoring force is related to the constitutive equation in the presence of viscous damping as follows

\begin{equation*}
F_{res} = -c\dot{x} -EA_c\frac{x}{L}+DA_c\left|\frac{x}{L}\right|\frac{x}{L}\,,
\end{equation*}
where $\dot{x}=dx/dt$ and $A_c, L$ denote the cross-sectional area of the graphene sheet and the length of the graphene sheet, respectively. The moving plate is subject to an electrostatic Coulomb force expressed in terms of a voltage function $V(t)$ as follows:
\begin{equation*}
F_{C} = \frac{\varepsilon_0 AV(t)^2}{2(d-x)^2}\,.
\end{equation*}
Using Newton's second law of motion we thus arrive at the following nonlinear differential equation
\[
m \ddot{x}=F_{res}+F_{C}\,,
\]
where $\ddot{x}=d^2x/dt^2$, hence
\begin{equation}\label{eq_dim}
m\ddot{x}+c\dot{x}+EA_c\frac{x}{L}-DA_c\left|\frac{x}{L}\right|\frac{x}{L}=\frac{\varepsilon_0 AV(t)^2}{2(d-x)^2}.
\end{equation}

Numerical simulations \cite{WKK17} show that the solutions of \eqref{eq_dim} differ significantly depending on whether a nonlinear stiffness constant $D$ is included or not. Recently, the pull-in instability was studied for an undamped ($c=0$) mass-spring system with nonlinear stiffness when the voltage is kept constant. Indeed, in \cite{SKNW19}, the authors derived necessary and sufficient conditions for the occurrence of dynamic pull-in and established an operational diagram for a device made of graphene.

It is easy to see that \eqref{eqn:main} is equivalent to our graphene model \eqref{eq_dim}. Indeed, if we substitute in the equation \eqref{eq_dim}
\begin{equation*}
\frac{x}{d}\mapsto x\,,\quad t\sqrt{\frac{EA_c}{mL}} \mapsto t\,,\quad c\sqrt{\frac{L}{mEA_c}}\mapsto c\,,\quad V(t)\sqrt{\frac{\varepsilon_0 AL}{2EA_cd^3}}\mapsto  V(t)\,,
\end{equation*}
and define the parameter
\begin{equation*}
\alpha=\frac{Dd}{EL}
\end{equation*}
we obtain the following dimensionless form for the graphene model
\begin{equation}\label{eqn:less}
\ddot{x}+c \dot{x} +x-\alpha |x|x = \frac{V^2(t)}{(1-x)^2}\,,
\end{equation}
where $c\ge 0$, $\alpha \ge 0$. Thus, with $V(t)=\lambda\bigl(1+\delta\cos{(\omega t)}\bigr)$ and $h(x)=x(1-\alpha |x|)$ in \eqref{eqn:main}, we recover the graphene model \eqref{eqn:less}.\\
%============================================================================================================================
% Pull-in instability
%==================================================================================
\subsection{Analysis of the pull-in instability}
%======================================================================================================================
In this Subsection we carry out a rigorous analysis of the existence of $T-$ periodic solutions and estimate the associated pull-in voltage for our system.\\ 
\begin{theorem} \label{thm:Aalpha}
	Let $\mu_\alpha=\sqrt{4\alpha^2-4\alpha+9}$ and 
	\begin{equation}\label{A_alpha}
	A_{\alpha}= \frac{(2\alpha +3 -\mu_\alpha)(5-2\alpha +\mu_\alpha)(6\alpha -3 +\mu_\alpha)^2}{4096\alpha^3}\,.
	\end{equation}
	If $ V_M^2\le  A_{\alpha}$, the system \eqref{eqn:less} admits a periodic solution with the same period $T$ as $V(t)$.\\ 
\end{theorem}
\begin{proof}
	Let us define 
	\[
	f(x)=(x-\alpha x|x|)(1-x)^2
	\] and 
	\[
	\beta=\min(1,1/\alpha)\,.
	\] 
	Then, $f(x)=(x-\alpha x^2)(1-x)^2 $ on $ [0,\beta]$, and one can easily find that 
	\[
	x_{1,2}=(3+2\alpha \pm \sqrt{9 - 4\alpha + 4\alpha^2})/(8\alpha)\quad\text{and}\quad x_3=1
	\] 
	represent the local extrema of $f(x)$, solving $f'(x)=0$. All three roots of $f'(x)=0$ are real as $9-4\alpha+4\alpha^2=8+(1-2\alpha)^2 >0$ and are positive. Since $ f(x)\geq 0 $ for $ x\in [0,\beta]$ and  $ f(0)=f(\beta)=0, $ the local maximum of the function $ f(x) $  on $ [0,\beta]$ occurs at the smallest non--negative critical point, that is, at $x_{1}=(3+2\alpha - \sqrt{9 - 4\alpha + 4\alpha^2})/(8\alpha).$ Furthermore, $ f(0)=0\le (V_M)^2\le A_{\alpha}=f(x_1) $ together with the Intermediate Value Theorem imply that $ f(x)= (V_M)^2$ has a root in $ [0,\beta] \subseteq [0,1].$ This shows that $ x-\alpha x|x|= (V_M)^2/(1-x)^2$ has a solution in $ [0,1]. $ One can then conclude based on Theorem \ref{thm:main} that \eqref{eqn:less} has a periodic solution with the same period as $V(t).$ This concludes the proof of the theorem. \end{proof}
\begin{remark}
	Notice that in the limiting case, i.e., $\alpha\rightarrow 0^+$, we have $A_\alpha = \frac{4}{27}$, which corresponds to the static pull-in voltage for the model with linear restoring force.\\[2ex]
\end{remark}
In the case of negligible damping, i.e., $c=0$, and a time-independent applied voltage, i.e., $\delta=0$ in Eq. \eqref{eq_ACDC}, it has been shown that the model equation \eqref{eqn:less} subject to zero initial conditions has periodic solutions if 
\[
\lambda^2 <\frac{(2\alpha+3-\tilde\mu_\alpha)(-4\alpha^2+24\alpha-9+2\alpha\tilde\mu_\alpha+3\tilde\mu_\alpha)}{648\alpha^2}
\]
where $\tilde\mu_\alpha=\sqrt{4\alpha^2-6\alpha+9}$ \cite{SKNW19}. Notice that Theorem~\ref{thm:Aalpha} does not specify the initial conditions for which the periodic solution exists. However, its amplitude can be estimated using the following\\
\begin{lemma}\label{lem:upp}
	If $x$ is a periodic solution of  \eqref{eqn:less} then either
	\begin{equation}\label{ine:1}
	\max_{t\ge 0} \left\lbrace x(t)\right\rbrace  \le -1/\alpha	
	\end{equation}
	or
	\begin{equation} \label{ine:2}
	0\le \max_{t\ge 0} \left\lbrace x(t)\right\rbrace \le  \min \left\lbrace 1/\alpha, 1-2V_m\sqrt{\alpha} \right\rbrace.
	\end{equation}
\end{lemma}
\begin{proof}	
	Following the same argument as in the second part of Theorem \ref{thm:main}, we say that $ x(t) $ has a maximum at $ t_1\in[0,T]. $ Thus, it is clear that 
	$\ddot x(t_1) \le 0$ and $\dot x(t_1)=0.$ Furthermore, from \eqref{eqn:less} it follows that
	\begin{equation}\label{e2}
	x(t_1)\bigl(1-\alpha |x(t_1)|\bigr)\bigl(1-x(t_1)\bigr)^2 \ge V^2(t_1) \ge V_m^2\ge0. 
	\end{equation}
	Now, we consider three cases:\\[2ex]
	\textit{Case 1:} $ |x(t_1)|> 1/\alpha$.\\
	Then, inequality \eqref{e2} implies that  $ x(t_1)\le 0. $ Thus, $x(t)\le x(t_1)< -1/\alpha$.\\[2ex]
	\textit{Case 2:} $ |x(t_1)| < 1/\alpha$.\\
	Then, inequality \eqref{e2} implies that  $ x(t_1) \geq 0 $ and $ x(t) \le 1/\alpha .$ On the other hand, by means of the inequality 
	$$ 4\alpha x(t_1)\bigl(1-\alpha x(t_1)\bigr) \le \bigl(\alpha x(t_1)+1-\alpha x(t_1)\bigr)^2=1$$
	we get $ \displaystyle x(t_1)\bigl(1-\alpha x(t_1)\bigr) \le \frac{1}{4\alpha}$. Next, from \eqref{e2} we have that 
	\begin{align*}
	 V_m^2 &\le x(t_1)\,\,\bigl(1-\alpha |x(t_1)|\bigr)\bigl(1-x(t_1)\bigr)^2 \\ &= x(t_1)\bigl(1-\alpha x(t_1)\bigr)\bigl(1-x(t_1)\bigr)^2 
	 \le \frac{1}{4\alpha}\bigl(1-x(t_1)\bigr)^2.
	\end{align*}
	So, $2V_m\sqrt{\alpha} \le 1-x(t_1) $, which implies that 
	$$ x(t)\le x(t_1)\le 1-2V_m\sqrt{\alpha} $$ and $ 2V_m\sqrt{\alpha} \le 1$.\\[2ex]
	\textit{Case 3:} $ |x(t_1)| =1/\alpha$.\\ 
	If $x(t_1)=-1/\alpha$ then \eqref{ine:1} is true since $x(t) \le x(t_1)$ for all $t$. If $x(t_1)=1/\alpha$, then \eqref{e2} implies $V_m=0$ in which case \eqref{ine:2} holds true since all physical solutions are less than $1$.
\end{proof}
%===================================================================================

%===================================================================================
We note that in general, the case \eqref{ine:1} of Lemma~\ref{lem:upp} cannot be avoided. Indeed, let us take $V(t)=V$ constant and consider the function $f(x)=x-\alpha x |x|-V^2/(1-x)^2$. Note that over the interval $(-\infty,-1/\alpha]$ the function is continuous and hence by the Intermediate Value Theorem it has a root since $f(-1/\alpha)<0$ and $\lim_{x \to -\infty } f(x) =\infty$ as far as $\alpha>0$. If $x_0$ is one such root, clearly $x(t)=x_0$ is a constant (hence periodic) solution to \eqref{eqn:less} and \eqref{ine:1} holds. However, it is an interesting problem to investigate if \eqref{ine:1} can be avoided when $V(t)$ is not constant but oscillates periodically in time.\\

\begin{lemma}\label{lem:low}
	If $x(t)$ is a periodic solution of  \eqref{eqn:less} then either
	\begin{equation}\label{ine:3}
	\min_{t\ge 0} \left\lbrace x(t)\right\rbrace  \ge -\frac{\sqrt{T}}{2c} \left( \frac{\alpha V_M}{1+\alpha}\right)^{2}+ x(t_0)
	\end{equation}
	or
	\begin{equation} \label{ine:4}
	\min_{t\ge 0} \left\lbrace x(t)\right\rbrace \ge  \max \left\lbrace -\frac{\sqrt{T}}{2c}  \left( \frac{\alpha V_M}{\alpha - 1}\right)^{2}+ x(t_0), -\frac{\sqrt{T}}{8\alpha c}  \left( \frac{ V_M}{V_m}\right)^{2}+ x(t_0)\right\rbrace,  
	\end{equation}
	where $ \displaystyle x(t_0) \ge -\frac{1}{\alpha} $ or $ \displaystyle x(t_0) \ge \frac{-1-\sqrt{1+4\alpha V_M^2}}{2\alpha}$.\\
\end{lemma}
\begin{proof}
	We assume that there exists $ t_0 $ such that $ \displaystyle x(t_0) \ge -\frac{1}{\alpha}. $ Otherwise, integrating \eqref{eqn:less} over a period, one can see that 
	\begin{align*}
	TV_M^2 &\ge \int_{0}^{T}\frac{V^2(s)}{\bigl(1-x(s)\bigr)^2}ds = \int_{0}^{T} x(s)\bigl(1-\alpha |x(s)|\bigr)\,ds\\ &=\int_{0}^{T} x(s)\bigl(1+\alpha x(s)\bigr)\,ds.
	\end{align*} 	
	The last inequality implies that there is a $ t_0 $ such that
	\begin{equation}\label{e3}
	x(t_0)\bigl(1+\alpha x(t_0)\bigr) \le V_M^2. 
	\end{equation} 
	Solving \eqref{e3} yields $ \displaystyle x(t_0) \ge \frac{-1-\sqrt{1+4\alpha V_M^2}}{2\alpha}. $
	
	On the other hand, multiplying \eqref{eqn:less} by $\dot{x}$ and integrating over a period, we arrive at
	\begin{equation}\label{e4}
	c \int_0^T\bigl(\dot{x}(s)\bigr)^2\,ds  = \int_{0}^{T}\frac{V^2(s)\,\dot{x}(s)}{\bigl(1-x(s)\bigr)^2}\,ds.
	\end{equation} 	
	The last relation together with Cauchy$-$Schwartz inequality implies that
	\begin{equation*} 
	c \int_0^T\bigl(\dot{x}(s)\bigr)^2\,ds   \le V_M^2 \left(\int_0^T\bigl(\dot{x}(s)\bigr)^2\,ds\right)^{1/2}\, \left( \int_{0}^{T}\frac{1}{\bigl(1-x(s)\bigr)^4}\,ds\right)^{1/2}. 
	\end{equation*}
	Thus, by Lemma \ref{lem:upp} one can easily show that either 
	\begin{equation}\label{e5}  
	\left(\int_0^T\bigl(\dot{x}(s)\bigr)^2\,ds\right)^{1/2}  \le \frac{1}{c}  \left( \frac{\alpha V_M}{1+\alpha}\right)^{2} 
	\end{equation}
	or 
	\begin{equation}\label{e6}
	\left(\int_0^T\bigl(\dot{x}(s)\bigr)^2\,ds\right)^{1/2}  \le \min \left\lbrace \frac{1}{c}  \left( \frac{\alpha V_M}{\alpha - 1}\right)^{2}, \frac{1}{4\alpha c}  \left( \frac{ V_M}{V_m}\right)^{2} \right\rbrace\,. 
	\end{equation}
	Employing the Cauchy--Schwartz inequality one more time, yields
	\begin{equation}\label{e7}
	\begin{split}
	x(t)-x(t_0) &= \int_{t_0}^{t}\dot{x}(s)ds \ge -\int_{0}^{T}\left[ \dot{x}(s)\right]^{-} ds= -\frac{1}{2} \int_0^T|\dot{x}(s)|\,ds \\
	&\ge -\frac{\sqrt{T}}{2}
	\left(\int_0^T\bigl(\dot{x}(s)\bigr)^2\,ds\right)^{1/2}\,.
	\end{split}
	\end{equation} 	
	Now, \eqref{e5}$-$\eqref{e7} imply that either
	\[
	\displaystyle x(t) \ge -\frac{\sqrt{T}}{2c}  \left( \frac{\alpha V_M}{1+\alpha}\right)^{2}+ x(t_0)
	\] 
	or
	\[
	\displaystyle x(t) \ge \max \left\lbrace -\frac{\sqrt{T}}{2c}  \left( \frac{\alpha V_M}{\alpha - 1}\right)^{2}+ x(t_0), -\frac{\sqrt{T}}{8\alpha c}  \left( \frac{ V_M}{V_m}\right)^{2}+ x(t_0)\right\rbrace
	\] 
	which proves the assertion. 
\end{proof}
Let us now consider Eq. \eqref{eqn:less} subject to zero initial conditions. In the next lemma we determine a lower bound for the period of the positive oscillations.\\
\begin{lemma}\label{lemma_T}
	Let $x\ge 0$ be a periodic solution of the initial value problem for Eq. \eqref{eqn:less} with zero initial conditions $x(0)=x'(0)=0$. Then, it holds true that the period $T$ of the solution $x(t)$ satisfies
	\begin{equation}\label{T_pi}
	T\ge \pi\,.
	\end{equation}
\end{lemma}
\begin{proof}
	Multiplying Eq. \eqref{eqn:less} by $x(t)$ and integrating over the period $T$, yields
	\begin{align}\label{eq_x_dot_mult}
	\int_0^T \bigl(\dot{x}(t)\bigr)^2\,dt &= \int_0^T x^2(t)\,dt- \alpha\int_0^T x^2(t)|x(t)|\,dt \nonumber\\ &-\int_0^T \frac{V^2(t)\, x(t)}{\bigl(1-x(t)\bigr)^2}\,dt\,.
	\end{align}
	By virtue of Wirtinger's inequality
	\[
	\pi^2\int_0^b f^2(t)\,dt\le b^2\int_0^b \bigl(f'(t)\bigr)^2\,dt
	\]
	which holds for all continuously differentiable functions $f(t)$ on $[0,b]$ with $f(0)=f(b)=0$, it follows from \eqref{eq_x_dot_mult} that
	\begin{equation}\label{T_pi_proof}
	\left(\frac{\pi^2}{T^2}-1\right)\int_0^T x^2(t)\,dt\le -\alpha\int_0^T |x(t)|^3\,dt-\int_0^T \frac{V^2(t)\, x(t)}{\bigl(1-x(t)\bigr)^2}\,dt\,.
	\end{equation}
	If $x\ge 0$, then necessarily
	\begin{equation*}
	\frac{\pi^2}{T^2}-1\le 0
	\end{equation*}
	since the right hand side of \eqref{T_pi_proof} is negative. This proves the statement of the lemma.
\end{proof}
Notice that Lemma~\ref{lemma_T} states that the zero initial value problem for Eq. \eqref{eqn:less} 
does not posses positive periodic solutions of period less than $\pi$.\\ 

In the next lemma, we prove again a theorem from \cite{YZZ2012} and show that the pull-in occurs if the applied voltage is sufficiently large.\\
\begin{lemma}
	Consider physically meaningful solutions $x(t)$ of Eq.~\eqref{eqn:less} such that $x(t)\in [0,1]$. If the voltage parameter $\lambda>0$ is sufficiently large, then the pull-in occurs at a finite time $t_p>0$ depending on $\lambda$ such that
	\[
	\lim\limits_{t\to t_p^-} x(t) = 1\quad\text{and}\quad \dot{x}(t)>0\quad\text{for all}\quad t\in (0,t_p)\,.
	\]
\end{lemma}
\begin{proof}
	Let us consider first the case $c>0$. Multiplying both sides of Eq.~\eqref{eqn:less} by $e^{ct}$ and integrating over the time interval $[0,t]$ yields
	\begin{equation*}
	e^{ct}\dot{x}(t) = \int\limits_0^t\left(\frac{\lambda^2(1+\delta\cos{(\omega s)})^2}{\bigl(1-x(s)\bigr)^2} -\bigl(x(s)-\alpha |x(s)|x(s)\bigr)\right)e^{cs}\,ds
	\end{equation*}
	and consequently
	\begin{equation*}
	\dot{x}(t) \ge \left(\lambda^2(1-\delta)^2-(1+\alpha)\right)\frac{1-e^{-ct}}{c}
	\end{equation*}
	due to the assumption $x\in [0,1]$. Integrating once more yields
	\begin{equation*}
	x(t)\ge \frac{1}{c}\left(\lambda^2(1-\delta)^2 - (1+\alpha) \right)\left(t-\frac{1-e^{-ct}}{c}\right) 
	\end{equation*}
	from which we conclude that there exists a finite time $t_p>0$ such that $x(t_p)=\lim\limits_{t\to t_p^-} x(t) = 1$ provided that $\lambda>\frac{\sqrt{1+\alpha}}{1-\delta}$.
	
	In the case of $c=0$, integrating both sides of Eq. \eqref{eqn:less} with respect to $t$ yields $\dot{x}(t)\ge \bigl(\lambda^2(1-\delta)^2 -(1+\alpha)\bigr)t$. Therefore, $x(t)\ge \bigl(\lambda^2(1-\delta)^2 -(1+\alpha)\bigr)\frac{t^2}{2}+x(0)$. This clearly proves the assertion as in the previous case. 
\end{proof}
%Thus, the zero initial value problem with the forcing voltage of period less than $\pi$ does not posses oscillatory solutions  

\section{A Poincar{\'e} map study: Periodic solutions with higher period}\label{Poincare}

\noindent The results obtained so far have been entirely analytical and concern conditions for the existence of $T-$ periodic solutions of our non-autonomous oscillators as well as rigorous estimates of their amplitude. However, they give us no information about: (a) the {\it local} stability of these solutions under small perturbations or, in case they are stable, (b) the {\it size} of the regions of bounded motion around them in the $\left(x(t),\dot{x}(t)\right)$ phase space. To find out, we turn in this Section to the computation of Poincar{\'e} maps of the solutions intersecting the plane $\left(x(t_k),\dot{x}(t_k)\right)$ at $t_k=kT$, $k=1,2,3,\ldots$ time intervals. 

As is well--known \cite{strogatz2018nonlinear,wiggins2003introduction}, every initial condition $\left(x(0),\dot{x}(0)\right)$ of our 3--dimensional flow yields a unique orbit on the Poincar{\'e} map. Thus, computing solutions for a large selection of initial conditions, yields an excellent view of the {\it global} properties of the flow and offers valuable information about the system's {\it global} dynamics. Let us illustrate this by applying the Poincar{\'e} map method to analyze the behavior exhibited by a graphene--MEMS model described by the equation:
\begin{equation}\label{model}
\ddot{x} + c \dot{x} + x - \alpha x \lvert x \rvert = \dfrac{V^2(t)}{(1-x)^2},
\end{equation}
with $V^2(t) = \left(V_d + f\, \cos{(\omega t)}\right)^2$. Note that eq. \eqref{model} has dimensionless form, $c$ is the damping parameter, $\alpha$ is the nonlinearity parameter, and $f$ and $\omega$ denote the forcing amplitude and frequency, respectively. 

Below, we reveal the rich organization of periodic solutions of eq. \eqref{model} in phase space, emphasizing the presence of large islands about them in the conservative case $c=0$, noting how these islands can turn to {\it basins of attraction}, when $c>0$ but small. Although the flow is dominated by the central periodic orbit (of period $T=\frac{2\pi}{\omega}$) and its island of bounded motion, we discover that varying the forcing frequency leads to the appearance of stable periodic orbits of higher periods, with sizable islands around them, further away from the period $T-$ solution. The important consequence of this is that these islands extend the distance away from the $T-$ periodic orbit, within which solutions remain bounded avoiding the pull--in instability.

In what follows, we fix the parameters $\left(\alpha,V_d,f\right) = \left(0.5, 0.01, 0.3\right)$ and restrict our attention to the forcing frequency and the damping coefficient. To examine more clearly the effect of the forcing frequency, we depict in Fig. \ref{fig2} the Poincar\'{e} map (or Poincar\'{e}  Surface of Section (PSS) of the flow \cite{strogatz2018nonlinear}) for different values of $\omega \in \left\{1.2, 1.25, 1.3, 1.35\right\}$, with $c=0$. Evidently, while the central orbit is crucial for the system's stable dynamics, there exist many other periodic orbits with period $T^{\star}=nT$ around it, which are surrounded by islands that vary in size, see e.g. orbits of periods $4T$ or $7T$ in Fig. \ref{fig2} . Observe also that the $3T-$ periodic points, which have the largest islands around them,  move towards the origin as $\omega$ increases, and disappear approximately at $\omega = 1.38$. 

\begin{figure}[t]
	\centering
	\includegraphics[keepaspectratio, width=0.5\textwidth]{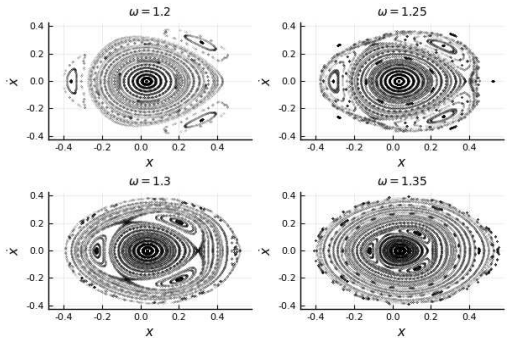}
	\caption{Poincar\'{e} Surface of Section for different values of $\omega \in \left\{1.2, 1.25, 1.3, 1.35\right\}$, with $c=0$.
		\label{fig2}}
\end{figure}

Of course, these $nT-$ periodic orbits are characterized by different amplitudes, as can be seen in Fig. \ref{fig3} for $\omega=1.3$. In particular, we record here the amplitudes of the solutions associated with initial conditions placed on a $500 \times 500$ grid of points over the interval $\left(-0.4, 0.55\right) \times \left(-0.4, 0.4\right)$. The non-colored points lead to unstable orbits that eventually run away to infinity. The  orbits corresponding to the colored points in Fig. \ref{fig3} are either periodic or quasi--periodic and hence exhibit bounded oscillations around the origin for all time.

\begin{figure}[ht]
	\centering
	\includegraphics[keepaspectratio, width=0.45\textwidth]{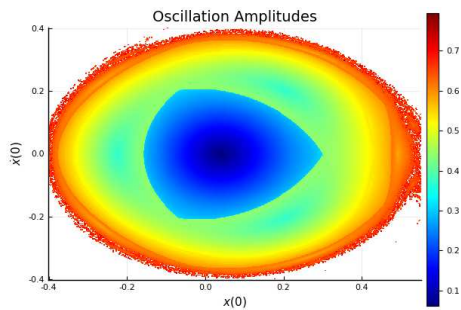}
	\caption{Amplitudes of the solutions for $\omega=1.3$ associated with an equally spaced $500 \times 500$ grid of initial conditions over the interval $\left(-0.4, 0.55\right) \times \left(-0.4, 0.4\right)$. \label{fig3}}
\end{figure}

It is interesting to examine, finally, the influence of the damping parameter on the properties of some of the higher--period orbits, when we set $\omega = 1.21$ (compare with Fig. \ref{fig4} for the undamped case) and restrict our attention to orbits with periods $3T, 4T, 7T$ and $18T$. Note that it is easy to identify the initial conditions leading to these orbits, using the PSS plot of the $c=0$ case shown in Fig. \ref{fig4}. 

\begin{figure}[h]
	\centering
	\includegraphics[keepaspectratio, width=0.5\textwidth]{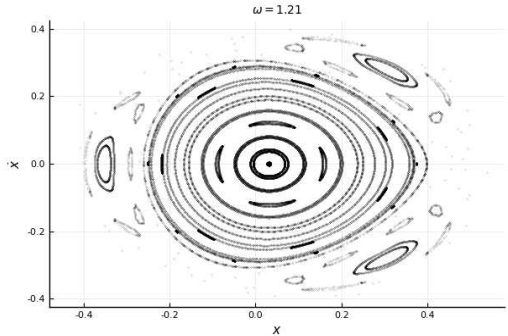}
	\caption{PSS for $c=0$ and $\omega=1.21$.  The orbits with period $3T, 4T, 7T$ and $18T$ can be easily identified in the plot by counting the corresponding chains of islands. \label{fig4}}
\end{figure}

Observe that the periodic orbits on these PSS come in pairs of one stable (with islands around it) and one unstable located at saddle points in between, with thin regions of chaotic behavior around them. When the damping exceeds a certain threshold, these stable--unstable pairs annihilate each other, depending on their rotation number and $\omega$. This happens via saddle--node bifurcations \cite{strogatz2018nonlinear,wiggins2003introduction}, through which the stable and unstable orbits coalesce and disappear. Thus, by slowly increasing $c$ one can identify a critical value of damping $c^{\star}$ at which these periodic orbits vanish. For example, such critical values $c^{\star}$, corresponding to $\omega=1.21$ and different periods $T^{\star} \in \left\{3T, 4T, 7T, 10T\right\}$, are $1.15 \times 10^{-2}$, $7.75\times 10^{-5}$, $1.5\times 10^{-5}$ and $6.8\times 10^{-6}$, respectively. 

Note first that the period of these orbits is inversely proportional to the value of $c^{\star}$ at which they disappear. Furthermore, the corresponding $c^{\star}$ values are also related to the forcing frequency. Indicatively, if we set $\omega=1.3$, the $3T$--periodic orbit ceases to exist for $c\approx 6.3\times10^{-3}$, which is substantially lower than the critical damping level associated with the $3T$--periodic orbit for $\omega=1.21$. 

Still, before these orbits disappear, it is interesting to examine their basins of attraction. Consider for example, the central period $T$ point at the origin and the $3T$--periodic orbit in Fig. \ref{fig7} with $c=10^{-3}$. Starting with a $500\times 500$ grid of initial conditions  over $\left(-0.41,0.55\right) \times \left(-0.41,0.41\right)$ and marking with blue and red colors respectively the initial conditions leading to the $T$--period and $3T$--period orbits, reveals a remarkable picture of intertwined basins spiraling outward from the origin and leading to a different attractor corresponding to the different color in each case.

\begin{figure}[h]
	\centering
	\includegraphics[keepaspectratio, width=0.45\textwidth]{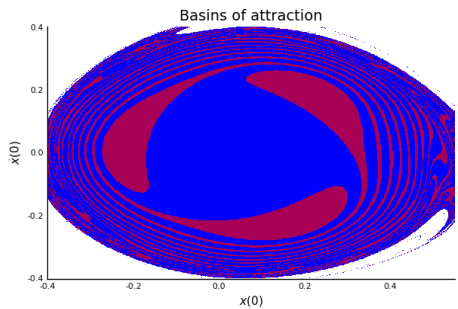}
	
	\caption{Basins of attraction for the central $T$--period (blue) and and the $3T$--period (red) orbits, for $\left(c,\omega\right) = \left(10^{-3},1.3\right)$. The white regions correspond to unstable behavior.
		\label{fig7}}
\end{figure}

\section{Conclusions}\label{conclusions}

In the present paper we undertook a thorough analytical and numerical study of the dynamical properties of a micro-electro-mechanical system (MEMS) commonly used as an electrostatic actuator. It consists of a 1-dimensional mass-spring model with two parallel capacitor plates, one of which is attached to a spring, while the other is fixed. The system is non-autonomous, with a $T-$ periodic voltage applied to the capacitor of the form $V(t)=V_{dc}+V_{ac} \cos(\omega t)$, with $\omega=2\pi/T$. Its equation of motion, 
\begin{equation*}
m \ddot{x}(t)+c \dot{x}(t)+h(x)=\frac{\varepsilon_0A V^2(t)}{2 \bigl(d-x(t)\bigr)^2}\,
\end{equation*}
includes a stiffness function $h(x)$ and a viscous dissipation force linearly proportional to the speed of the moving plate. Most importantly, it contains a singularity at $x(t)=d$, known as the pull--in instability, which causes the solutions to diverge, if $x(t)$ reaches a value sufficiently close to $d$. Therefore, our main purpose in this paper was to study periodic solutions of the above equation with large regions of initial conditions around them, for which the motion remains confined as $t\rightarrow \infty$ avoiding the pull-in instability.  

Our first step was to obtain analytical results concerning the existence of periodic solutions with the same period $T$ as the forcing term. They are known to be generally stable with large domains of initial conditions around them where the motion can be periodic, quasiperiodic or even chaotic. Indeed, we were able to estimate pull-in voltages for this model for a wide class of nonlinear stiffness functions $h(x)$. Next, we considered the model of a MEMS with  one of its capacitor plates attached to a spring made of graphene, where $h(x)\propto x|x|$ and obtained rigorous results concerning its $T-$ periodic solutions in connection with the initial conditions and parameter values of the problem.

Our analytical results, however, cannot ascertain the stability of these fundamental solutions, either locally or globally. Local stability, of course, can always be studied by linearization and application of Floquet theory, which involves the expansion of the periodic solution in Fourier series and the study of the eigenvalues of an (infinite) Hill's matrix. However, since we are ultimately interested in locating phase space regions where the motion is bounded  and the pull--in instability is avoided, we decided to integrate numerically the differential equations for the graphene system and reveal directly local and global stability properties of the system plotting the orbits on a ($T-$ periodic) Poincar{\'e} map of the flow. 

Thus, we were able to obtain valuable information about the system's dynamics, for a wide range of parameter values: First, in the absence of damping, we graphically depicted large regions of bounded oscillations around the $T-$ periodic solution located at the center of the 2--dimensional Poincar{\'e} surface of section. Secondly, we demonstrated the existence of stable (longer) periodic orbits of period $nT, n>1$, whose islands of bounded motion further extend the domain of stable oscillations of the model. Finally, in the presence of small dissipation, we displayed large basins of attraction around these fundamental solutions and obtained maximal sets of initial conditions for which the pull--instability is entirely avoided.

\subsection*{Acknowledgments}

SK acknowledges the support from a grant from the Ministry of Education and Science of the Republic
of Kazakhstan within the framework of the Project AP08051987. KK acknowledges useful discussions with Professor Christos Spitas and partial support for this work by funds from the Ministry of Education and Science of Kazakhstan, in the context of the Nazarbayev University internal grant HYST (2018-2021). AK was supported in part by Nazarbayev University FDCR Grants N 090118FD5353

%%%%%%%%%%%%%%%%%%%%%%%%%%%%%%%%%%%%%%%%%%%%%%
%\bibliographystyle{elsarticle-num}
%\bibliography{Bibfile}
%%%%%%%%%%%%%%%%%%%%%%%%%%%%%%%%%%%%%%%%%%%%%%

\end{document}